\documentclass[final]{dmtcs-episciences}

\usepackage[utf8]{inputenc}
\usepackage{subfigure}

\usepackage[round]{natbib}

\usepackage{amsthm}
\usepackage{mathtools}

\theoremstyle{plain}
	\newtheorem{theorem}{Theorem}
	\newtheorem{corollary}[theorem]{Corollary}
	\newtheorem{lemma}[theorem]{Lemma}
	\newtheorem{proposition}[theorem]{Proposition}

\theoremstyle{remark}
	\newtheorem{remark}[theorem]{Remark}

\def\Enn{\mathbb{N}}
\DeclarePairedDelimiter\floor{\lfloor}{\rfloor}

\author[Jeffrey Shallit and Ingrid Vukusic]{Jeffrey Shallit\affiliationmark{1}\thanks{Research supported by NSERC grant 2024-03725.}
  \and Ingrid Vukusic \affiliationmark{2}\thanks{Research funded by the Austrian Science Fund (FWF) 10.55776/J4850.}
}

\title[Fibonacci word rectangles]{Balanced Fibonacci word rectangles,\\
and beyond}

\affiliation{
  School of Computer Science, University of Waterloo, Waterloo, Canada\\
  Department of Mathematics, University of York, York, United Kingdom
}

\begin{document}
\publicationdata{vol. 28:2}{2026}{26}{10.46298/dmtcs.16955}{2025-11-19; 2025-11-19; 2026-03-31}{2026-04-09}
\maketitle
\begin{abstract}
Following a recent paper of Anselmo et al.,
we consider $m \times n$ rectangular matrices formed from the Fibonacci word, and
we show that their balance properties can be solved with a finite automaton.
We also generalize the result to every Sturmian characteristic word corresponding
to a quadratic irrational.  Finally, we also examine the analogous question for the Tribonacci word and the Thue--Morse word.
\end{abstract}

\section{Introduction}

Let $(a_i)_{i \geq 0}$ be an infinite sequence over $\Enn$.
We can define an infinite matrix ${\cal A} =(a_{k,\ell})_{k,\ell\geq 0}$ by $a_{k,\ell} = a_{k+\ell}$.  
Then we can view the $m \times n$ submatrices of $\cal A$ as matrices where the successive rows form shifted length-$n$ ``windows'' of $(a_i)_{i \geq 0}$, and
the entry in the upper left corner has index sum $i$:
\begin{equation}
    A(i,m,n)
    := \begin{pmatrix}
        a_i & a_{i+1} & \dots & a_{i+n-1} \\
        a_{i+1} & a_{i+2} & \dots & a_{i+n} \\
        \vdots & \vdots &     & \vdots \\
        a_{i+m-1} & a_{i+m} & \dots & a_{i+m+n-2} \\
    \end{pmatrix}.
\label{aimn}
\end{equation}
These matrices are sometimes called Hankel matrices in the literature \citep{Markovsky:2019}.

The sum over all entries in the matrix $A(i,m,n)$ equals
\begin{equation}
    T(i,m,n)
    := \sum_{k = 0}^{m-1} \sum_{\ell = 0}^{n-1} a_{i+k+\ell}.
\label{sum}
\end{equation}
It is clear that $A(i,m,n)$ and $A(i,n,m)$ contain the same number of occurrences of each element of $(a_i)_{i \geq 0}$, and so in particular $T(i,m,n) = T(i,n,m)$ for all $i,m,n$.

Now suppose that $(a_i)_{i \geq 0}$ is an infinite sequence (or word) over $\{0,1 \}$.  In this case we call $A(i,m,n)$ a {\it word rectangle}.  We say the $m \times n$ rectangles are
{\it balanced with respect to each other\/} (or just that the pair $(m,n)$ is {\it balanced}\,) if there are at most two distinct values among all the sums $(T(i,m,n))_{i \geq 0}$.  This is a refinement of a more familiar notion of balance for
binary words, one that has been
well-studied \citep{Vuillon:2003}.  Our notion of two-dimensional balance appears under a different name in work by \cite{Puzynina:2019}.

From results of \cite{Berthe} we know that $(m,n)$ is balanced for {\it all\/} $m,n$ only in trivial cases.  Hence,
the {\it balance problem\/} for $(a_i)_{i \geq 0}$ is the following: determine  which pairs $(m,n)$
of natural numbers  are balanced. 

The famous infinite Fibonacci word ${\bf f} = f_0 f_1 f_2\cdots = 01001010\cdots$ has been the subject of numerous papers \citep{Berstel:1986b}.  It can be defined via the formula
\[
    f_{n-1} = \floor{(n+1)\gamma} - \floor{n \gamma} \in \{0,1\},
\]
for $n \geq 1$,
where $\gamma = (3- \sqrt{5})/2$.  
Recently \cite{Anselmo} showed the following (their Theorem 18):
for $m,n \geq 2$, if $\max(m,n)$ is a Fibonacci number, then the $m \times n$ word rectangles of the infinite Fibonacci matrix ${\cal A}_{\gamma}$ are balanced.  Also see \cite{Berthe}. However, there are still other $(m,n)$ pairs with the balance property, such as $(4,3)$.  This raises the natural question, which $(m,n)$ pairs are balanced for the Fibonacci word?  One can easily verify, using the fact that there are exactly $t+1$ distinct contiguous blocks of length $t$ in the Fibonacci word, that $(m,n)$ is balanced for $1 \leq m \leq n \leq 6$ except when $(m,n) \in\{ (2,4),(3,6),(4,4), (4,6), (6,6) \}$.

In this note we prove a theorem (Corollary~\ref{fibb}) characterizing exactly those pairs $(m,n)$ that are balanced.  Our characterization is in terms of finite automata.
The particular result of Anselmo et al.\ then follows as a special case.  More generally, we show how the same sort of characterization can be obtained for the matrices formed from any Sturmian word corresponding to a quadratic irrational $\alpha$, and we show how to implement this in the free software {\tt Walnut}.

We also examine the analogous question for the Tribonacci word and the Thue--Morse word.

\section{Sturmian words}
Let $\alpha \in (0,1)$ be an irrational number, and define
$$a_i = \lfloor (i+1) \alpha \rfloor - \lfloor i \alpha \rfloor$$
for $i \geq 1$.  The infinite word
$(a_i)_{i \geq 1}$ is sometimes called a Sturmian characteristic word \citep{BrownT:1993,Berstel&Seebold:2002}.  Although by convention Sturmian words are indexed starting at $1$, we can define $a_0 = 0$ without creating any new blocks that do not already appear in $\bf f$,
by a well-known result \cite[Prop.~2.1.18]{Berstel&Seebold:2002}.

We now define the infinite matrix ${\cal A}_{\alpha}$ from the terms of $(a_i)_{i \geq 1}$; we call ${\cal A}_{\alpha}$ a Sturmian matrix.
By telescoping cancellation we see that
$$\sum_{j \leq i \leq k} a_i =
\lfloor (k+1) \alpha \rfloor -
\lfloor j \alpha \rfloor .$$
Hence, from the definition
\eqref{sum}, we get
\begin{align}
    T(i+1,m,n) - T(i,m,n) 
    &= (a_{i+n} + \dots + a_{i+m+n-1})
        - (a_{i} + \dots + a_{i+m-1})\nonumber\\
    &= (\floor{(i+m+n)\alpha} - \floor{(i+n)\alpha})
        - (\floor{(i+m)\alpha} - \floor{i \alpha}).
        \label{fund}
\end{align}
Note that for all $x,y\geq 0$ we have $\floor{x+y} = \floor{x} + \floor{y} + \delta$ with $\delta \in \{0,1\}$.
Therefore,
$\floor{(i+m+n)\alpha} - \floor{(i+n)\alpha} = \floor{m \alpha} + \delta_1$
and $\floor{(i+m)\alpha} - \floor{i \alpha} = \floor{m \alpha} + \delta_2$, with $\delta_1, \delta_2 \in \{0,1\}$.
Thus,
\begin{equation}
    T(i+1,m,n) - T(i,m,n)
    = \delta_1 - \delta_2
    \in \{-1,0,1\}.
\label{bnd2}
\end{equation}
This fact can also be seen in another, perhaps simpler, way: since by Eq.~\eqref{fund} we know $T(i+1,m,n) - T(i,m,n)$ is the difference in the number of $1$'s in the two length-$m$ words
$a_{i+n} \cdots a_{i+m+n-1}$ and
$a_i \cdots a_{i+m-1}$, Eq.~\eqref{bnd2} follows immediately from the fact that all Sturmian words are balanced.

Define $\Delta(i,m,n) := T(i+1,m,n)-T(i,m,n)$.
The argument above then gives the following criterion for $A(i,m,n)$ to be balanced.
\begin{lemma}
If $(a_i)_{i \geq 0}$ is a Sturmian word corresponding to the irrational number $\alpha$, then in the matrix 
${\cal A}_\alpha = (a_{i,j})_{i,j\geq 0}$
the $m \times n$ blocks are balanced
if and only if the sequence
$\Delta(i,m,n)$
has no blocks of the form
$1,0,0,\ldots, 0,1$ or
$-1,0,0,\ldots,0, -1$.
\label{lemma1}
\end{lemma}

\begin{proof}
We claim that
$(\Delta(i,m,n))_{i \geq 0}$ has a block of the form $1,0,0,\ldots, 0,1$ or
$-1,0,0,\ldots,0, -1$ if and only if 
$(m,n)$ is not balanced.

Suppose it does have a block of the form $a,0,0,\ldots,0,a$ for $a \in \{ -1, 1\}$; say there are indices
$i < j$ such that
$\Delta(i,m,n) = a$,
$\Delta(j,m,n) = a$ and
$\Delta(k,m,n) = 0$ for $i<k<j$.
Then by telescoping cancellation
we have $T(j+1,m,n)= T(i,m,n) + 2a$, proving that $(m,n)$ is not balanced.

Now suppose $(m,n)$ is not balanced, meaning that there exist $i< j$
such that $|T(i,m,n) - T(j+1,m,n)| \geq 2$.  Without loss of generality we may assume that $j-i$ is as small as possible.
Since from Eq.~\eqref{bnd2} we know that $\Delta(i,m,n) \in \{-1,0,1\}$, it follows that we must have $\Delta(j) = a$,
$\Delta(i) = a$, and
$\Delta(k) = 0$ for $i < k < j$ and
$a \in \{-1,1\}$.
\end{proof}

\section{Finite automata}

Now suppose $\alpha \in (0,1)$ is a quadratic irrational.  Then, as shown by 
\cite{Schaeffer&Shallit&Zorcic:2024},
there is an algorithm to find a finite automaton $M$
that takes $n$ and $x$ in parallel as inputs, represented in a certain numeration system, and
accepts if and only if $x = \lfloor n \alpha \rfloor$.  The numeration system required is the so-called Ostrowski $\alpha$-numeration system, as discussed, for example, by \citet[\S 3.9]{Allouche&Shallit:2003}.

Furthermore, we also know that we can algorithmically translate a first-order logical formula $\cal F$ involving addition, subtraction, comparisons, and other automata, into an automaton accepting the values of the free variables making $\cal F$ true \citep{Bruyere&Hansel&Michaux&Villemaire:1994}.

One strategy that initially seems plausible is to construct an automaton that takes
$m,n,z$ in parallel as input, and accepts if $z$ is the number of distinct values of $T(i,m,n)$ occurring over all $i$.
If we had that, we could then accept those $(m,n)$ pairs corresponding to $z \leq 2$.  However, this is not possible; we show in Section~\ref{diverse} that
the number of distinct values taken by $T(i,m,n)$ for
$m = n = F_{3j}/2$ is $\Theta(j)$, where
$F_t$ is the $t$'th Fibonacci number. This is impossible for an automaton that computes a function $z$ from inputs $m,n$ and accepts all three in parallel \citep{Shallit:2021}.

Instead, we can use the balance criterion stated in Lemma~\ref{lemma1}. This criterion is clearly expressible as a formula $\cal F$ in first-order logic, since $x = \lfloor n \alpha \rfloor$ is expressible \citep{Schaeffer&Shallit&Zorcic:2024}.  Then, using the techniques discussed by \cite{Shallit:2023}, we can translate $\cal F$ into an appropriate automaton that takes $m$ and $n$ as input (in the Ostrowski $\alpha$-numeration system) and accepts if and only if the $m \times n$ blocks of the matrix $A$ are balanced.
We therefore get our main result.
\begin{theorem}
There is an algorithm that, given a quadratic irrational
$\alpha$ with $0 < \alpha < 1$, constructs a finite automaton that takes as input pairs $(m,n)$, represented in the Ostrowski $\alpha$-numeration system, and accepts if and only if the $m \times n$ blocks of the Sturmian matrix ${\cal A}_{\alpha}$  are balanced.
\label{main}
\end{theorem}

Let us now carry out this construction in {\tt Walnut} for the special case of the Fibonacci word, where
$\alpha = \gamma = (3-\sqrt{5})/2$.  We use the Zeckendorf numeration system \citep{Lekkerkerker:1952,Zeckendorf:1972}, which represents integers as a sum of distinct Fibonacci numbers $F_j$.  A representation of $n$ is written 
compactly as a binary
word $e_i \cdots e_2$, where
$n = \sum_{2 \leq j \leq i} e_j F_j$.

To implement the criterion in Lemma~\ref{lemma1}, we need to compute the difference
$T(i+1,m,n)-T(i,m,n) \in \{ -1,0, 1\}$.
Unfortunately we cannot write $-1$ in the ordinary Zeckendorf numeration system in {\tt Walnut}, since it only deals with natural numbers.\footnote{Actually, the negaFibonacci system is implemented in {\tt Walnut}, so we could have used that.  But we stick with the Zeckendorf representation system because it is more familiar.}
To get around this problem, we work with the quantity
$$T(i+1,m,n)-T(i,m,n)+1 \in \{0,1,2 \}$$
instead.

The following {\tt Walnut} code implements the criterion of Lemma~\ref{lemma1}.\footnote{For an introduction to {\tt Walnut}, see the Appendix.}  For more about {\tt Walnut}, see
\cite{Mousavi:2016,Shallit:2023}.
\begin{verbatim}
reg shift {0,1} {0,1} "([0,0]|[0,1][1,1]*[1,0])*":
# accepts x,y such that y is the left shift of the bits of x
def phin "?msd_fib (n=0 & y=0) | Ez $shift(n-1,z) & y=z+1":
# (n,y) is accepted iff floor(n*phi) = y, where phi = (1+sqrt(5))/2
def alpha "?msd_fib (n=0&z=0) | (n>0 & Ex $phin(n,x) & z+1+x=2*n)":
# z = floor(gamma*n), where gamma = (3-sqrt(5))/2
def diff "?msd_fib Ea,b,c,d $alpha(i,a) & $alpha(i+m,b) 
   & $alpha(i+n,c) & $alpha(i+m+n,d) & z+b+c=a+d+1":
# z = Delta(i,m,n) + 1 = T(i+1,m,n)-T(i,m,n)+1
def bal "?msd_fib ~Ei,j i>0 & i<j & (
   ($diff(i,m,n,2) 
      & (Ak (i<k & k<j) => $diff(k,m,n,1)) 
      & $diff(j,m,n,2)) |
   ($diff(i,m,n,0) 
      & (Ak (i<k & k<j) => $diff(k,m,n,1)) 
      & $diff(j,m,n,0)))":
# implements the balance criterion of Lemma 1
\end{verbatim}
This gives a $15$-state automaton {\tt bal} that accepts exactly those $(m,n)$ pairs, in the Zeckendorf numeration system, that are balanced.  An input consisting of a word made of pairs  is accepted  by the automaton 
if there is a path from state $0$ labeled with 
those pairs. See Figure~\ref{fig1}.
\begin{figure}[htb]
\includegraphics[width=\textwidth]{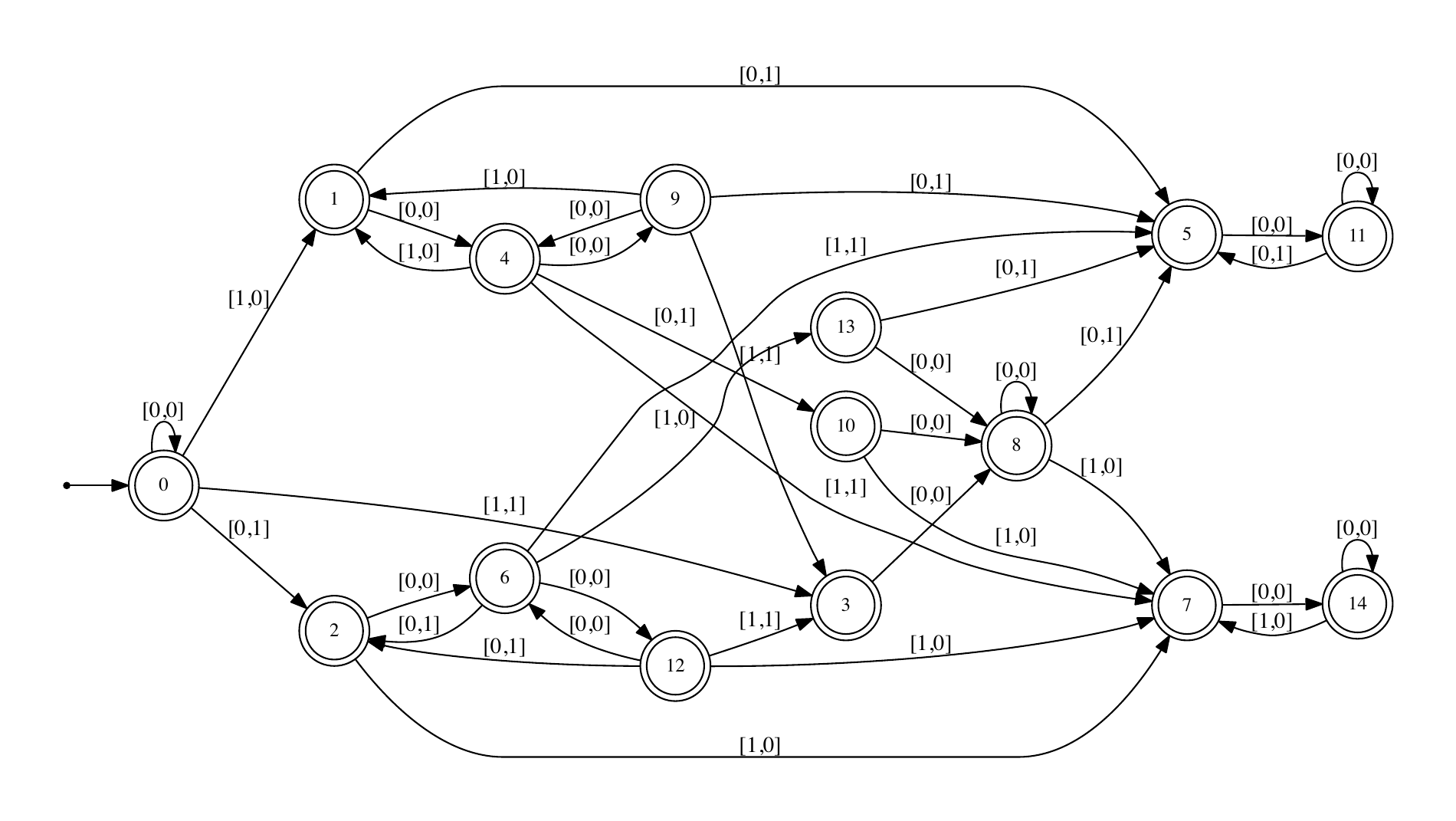}
\caption{Automaton accepting balanced $(m,n)$ pairs for the Fibonacci word.}
\label{fig1}
\end{figure}

We have therefore shown
\begin{corollary}
The automaton in Figure~\ref{fig1} solves the balance property for the Fibonacci word rectangles of dimension $m \times n$.
\label{fibb}
\end{corollary}

As an example, consider using the automaton to determine whether the $4 \times 18$ blocks are balanced.  We represent $4$
by $101$ in the Zeckendorf system and $18$ by
$101000$.  Padding the shorter with $0$ and entering the inputs in parallel
$[0,1][0,0][0,1][1,0][0,0][1,0]$ into the automaton in Figure~\ref{fig1}, we start at state $0$ and successively visit states $2, 6, 2, 7, 14$, and $7$ and accept.  So $(m,n)$ is balanced.

As a special case, we can now recover the theorem of Anselmo et al., as follows:
\begin{corollary}
If $\max(m,n)$ is a Fibonacci number then the $m \times n$ matrices of ${\cal A}_{\gamma}$
are balanced.
\end{corollary}
\begin{proof}
We use {\tt Walnut}:
\begin{verbatim}
reg isfib msd_fib "0*10*":
# is input a Fibonacci number
def max "?msd_fib (z=m|z=n) & z>=m & z>=n":
# is z = max(x,y)
eval anselmo "?msd_fib Am,n,z ($max(m,n,z) & $isfib(z)) => $bal(m,n)":
# translation of their theorem
\end{verbatim}
And {\tt Walnut} returns {\tt TRUE}.
\end{proof}

The {\tt bal} automaton allows us to prove a wide variety of results about the $(m,n)$-balanced pairs for ${\cal A}_{\gamma}$.  Here is one example.

\begin{proposition}
Let $n \geq 2$.
\begin{itemize}
  \item[(a)] 
  Let $F_{n+1} < i,i' < F_{n+2}$ and $j \geq F_{n-1}$.
  Then $(i,j)$ is balanced if and only if $(i',j)$ is balanced.
    \item[(b)]
    Let $1 \leq i < F_n$
    and $j \geq 1$.
    Then $(F_{n+1} + i,j)$
    is balanced if and only if
    $(F_{n+2} -i, j)$
    is balanced.
\end{itemize}
\end{proposition}

\begin{proof}
We use {\tt Walnut}.
For (a) we use the commands
\begin{verbatim}
reg adjfib msd_fib msd_fib "[0,0]*[0,1][1,0][0,0]*":
# accepts (F_n, F_{n+1}) for n>=2
eval claim_a "?msd_fib Au,v,x,i,ip,j ($adjfib(u,v) & x=u+v & i>v 
   & ip>v & i<x & ip<x & j>=v-u) => ($bal(i,j) <=> $bal(ip,j))":
# u = F_{n}, v = F_{n+1}, x = F_{n+2}
\end{verbatim}
And {\tt Walnut} returns {\tt TRUE}.

For (b) we use the commands
\begin{verbatim}
eval claim_b "?msd_fib Au,v,x,i,j 
   ($adjfib(u,v) & x=u+v & i>=1 & i<u & j>=1) 
   => ($bal(v+i,j) <=> $bal(x-i,j))":
# u = F_n, v = F_{n+1}, x = F_{n+2}
\end{verbatim}
And {\tt Walnut} returns {\tt TRUE}.
\end{proof}

We can also obtain a result about the density of balanced pairs $(m,n)$.
\begin{proposition}
Suppose $F_i < m \leq F_{i+1}$.  Then for all $n\geq 1$ there exists $j$,
$1 \leq j \leq F_{i+1}$, 
such that $(m,n+j)$
is balanced.   Furthermore the upper bound $F_{i+1}$ is optimal for each $m$.
\end{proposition}

\begin{proof}
We use {\tt Walnut}.
\begin{verbatim}
eval density "?msd_fib Am,x,y,n ($adjfib(x,y) & x<m & m<=y) 
   => Ej j>=1 & j<=y & $bal(m,n+j)":
# x = F_i, y = F_{i+1}
\end{verbatim}
which proves the first statement.  To prove the second, we show that for each $m$
there exists $n$ such that the least $j\geq 1$
with $(m,n+j)$ balanced is $\geq F_{i+1}$.
\begin{verbatim}
eval density2 "?msd_fib Am,x,y ($adjfib(x,y) & x<m & m<=y) 
   => (En Aj (j>=1 & j<y) => ~$bal(m,n+j))":
\end{verbatim}
And {\tt Walnut} returns {\tt TRUE}.
\end{proof}

\noindent These results essentially say that the density of $n$
for which $(m,n)$ is balanced is lower bounded by (approximately)
$\frac{1}{\varphi m}$.

\section{Another version of the criterion for balance}

We can obtain a more ``human readable'' version of the criterion for the $m \times n$ submatrices of ${\cal A}_{\gamma}$ to be balanced, as follows:  

\begin{theorem}
Assume $m \leq n$ and let the Zeckendorf expansion of $m$ (resp., $n$) be $F_{a_1} + \dots + F_{a_k}$ (resp., $ F_{b_1} + \dots + F_{b_\ell}$) for $a_1 > \cdots >a_k$ and $b_1 > \cdots > b_\ell$.
Then $(m,n)$ is accepted by {\tt bal} if and only if at least one of the following cases holds:
\begin{itemize}
    \item[(a)] $m = 0$ or $m = 1$; 

    \item[(b)] $m = F_{a_1}$ and $a_1 \notin \{b_1, \ldots, b_\ell\}$ and for $b_j > a_1 > b_{j+1}$ the indices $a_1$ and $b_j$ have the same parity; 
    
    \item[(c)] $m = F_{a_1}$ and $a_1 \in \{b_1, \ldots, b_\ell\}$; 
    
    \item[(d)] $a_1 =  b_\ell$, (i.e., the largest Fibonacci number in the representation of $m$ is the smallest Fibonacci number in the representation of $n$)
	and $b_\ell$ and $b_{\ell-1}$ have distinct parities; 
     
    \item[(e)] $a_1 < b_\ell$ (i.e., the largest Fibonacci number in the representation of $m$ is smaller than the smallest Fibonacci number in the representation of $n$). 
\end{itemize}
\end{theorem}

\begin{proof}
To find the characterization we traced the paths in the automaton in Figure~\ref{fig:colored}, which is the automaton {\tt bal} from Figure~\ref{fig1} restricted to $m \leq n$ using the
following {\tt Walnut} code:
\begin{verbatim}
def balp "?msd_fib $bal(m,n) & m<=n":
\end{verbatim}

\begin{figure}[htb]
\includegraphics[width=\textwidth]{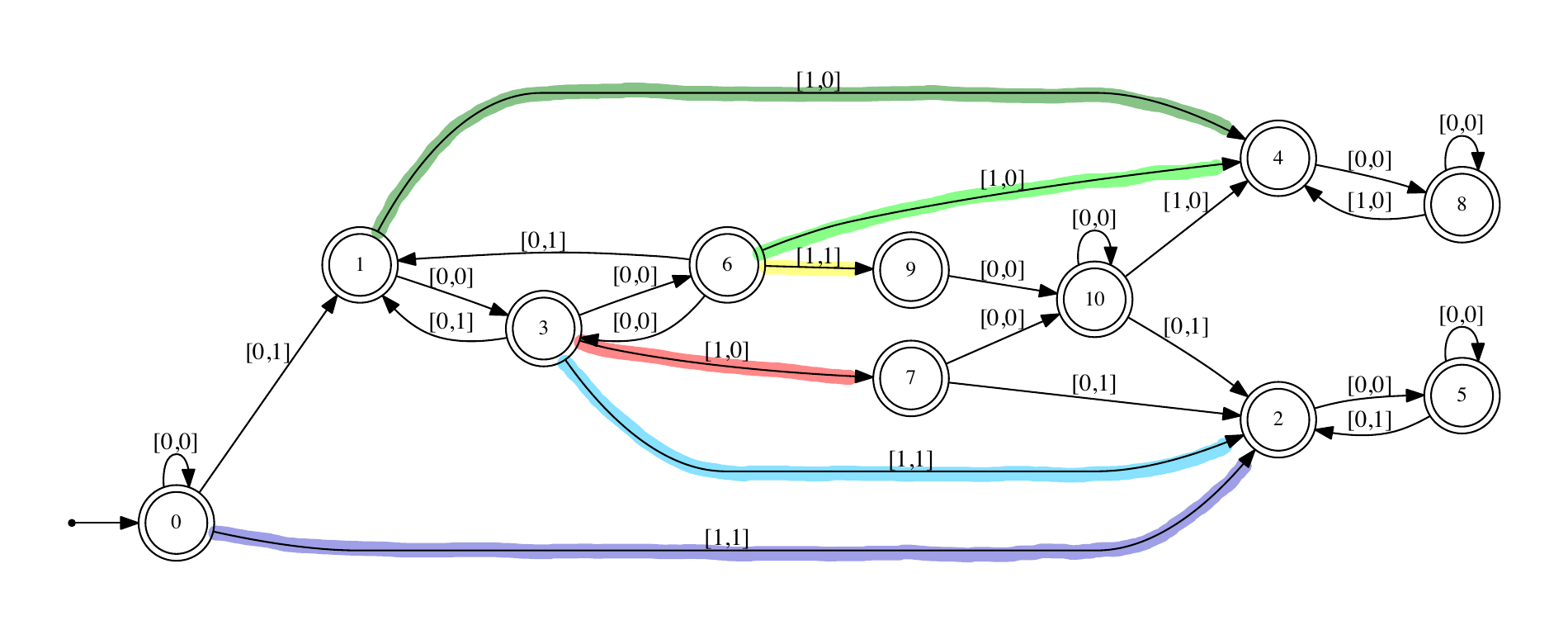}
\caption{Automaton {\tt balp} accepting balanced $(m,n)$ pairs for the Fibonacci word with $m \leq n$. }
\label{fig:colored}
\end{figure}

The {\tt Walnut} code below verifies the characterization. 
Note that for simplicity, the regular expressions are not necessarily valid Zeckendorf representations. We can do this because eventually in {\tt test} only the valid Zeckendorf representations are considered by
{\tt Walnut}.
The mentioned colors are meant to help the interested reader find the paths in Figure~\ref{fig:colored} corresponding to each regular expression. The ``upper part'' refers to the states $4,8$; the ``lower part'' to $2,5$ (and the state $10$ belongs to both).

\begin{verbatim}
reg case_a msd_fib msd_fib 
   "([0,0]|[0,1])*|([0,0]|[0,1])*([1,0]|[1,1])":

reg case_b msd_fib msd_fib 
   "([0,0]|[0,1])*[0,1][0,0]([0,0][0,0])*[1,0]([0,0]|[0,1])*":
# lower part of red

reg case_c msd_fib msd_fib "([0,0]|[0,1])*[1,1]([0,0]|[0,1])*":
# dark blue, light blue, and the lower part of yellow

reg case_d msd_fib msd_fib 
   "([0,0]|[0,1])*[0,1]([0,0][0,0])*[1,1]([0,0]|[1,0])*":
# upper part of yellow

reg case_e msd_fib msd_fib "([0,0]|[0,1])*([0,0]|[1,0])*":
# dark green, light green, and upper part of red

eval test "?msd_fib Am,n (m<=n) => 
   (($case_a(m,n)|$case_b(m,n)|$case_c(m,n)|$case_d(m,n)|$case_e(m,n)) 
   <=> $bal(m,n))":
\end{verbatim}
And {\tt Walnut} returns {\tt TRUE}.
\end{proof}

\section{Diverse rectangles}
\label{diverse}

\begin{theorem}
Consider the word rectangles formed from the Fibonacci word $\bf f$ with the indexing $f_0 f_1 f_2 \cdots\allowbreak = 0100\cdots$.
The number of distinct values
of $T(i,m,n)$, taken over all $i$,
for $m = n = F_{3k}/2$, is at least
$k+1$.
\label{bnd}
\end{theorem}

\begin{proof}
We claim that
\begin{align}
T (( F_{6n-1}-1)/4, F_{6n}/2, F_{6n}/2)  - T( (F_{6n+2}-1)/4, F_{6n}/2, F_{6n}/2) &= 2n \label{t1} \\
T((F_{6n+5}-1)/4, F_{6n+3}/2, F_{6n+3}/2) - T((F_{6n+2}-1)/4, F_{6n+3}/2, F_{6n+3}/2) &= 2n+1. \label{t2}
\end{align}

To prove these claims we use {\tt Walnut} and its capabilities of producing linear representations.   A linear representation for a function $f(n)$ is a triple $(v,\zeta, w)$ where $v$ is a row vector, $w$ is a column vector, and $\zeta$ is a matrix-valued morphism, such that $f(n) = v \zeta(x) w$ for all representations of $x$ that evaluate to $n$.  In this case, of course, $x$ is a Zeckendorf representation. 

Let us verify Eq.~\eqref{t1}.  We define three functions, as follows:
\begin{itemize}
    \item  $f(F_{6n}) = T( (F_{6n+2}-1)/4, F_{6n}/2, F_{6n}/2)$ for $n \geq 1$,
    and $f(t)= 0$ for all other $t$.

    \item  $g(F_{6n}) = T (( F_{6n-1}-1)/4, F_{6n}/2, F_{6n}/2)$ for $n \geq 1$,
    and $g(t)= 0$ for all other $t$.

    \item $h(F_{6n}) = 6n$
    for $n \geq 1$, and $h(t)= 0$ for all other $t$.
\end{itemize}
Then Eq.~\eqref{t1} is equivalent to the claim that $3f(t)-3g(t)-h(t) = 0$ for all $t$.  Our strategy is to construct linear representations for each of the functions $f(t)$, $g(t)$, and $h(t)$, using {\tt Walnut}, and then combine them to find a linear representation for
$3f(t)-3g(t)-h(t)$.

Consider the following code:
\begin{verbatim}
reg fib6 msd_fib "0*1(000000)*0000":
def cnt6a t "?msd_fib $fib6(t) & Ex,y $adjfib(t,x) & $adjfib(x,y) & 
   F[(y-1)/4+j+k]=@1 & j<t/2 & k<t/2":
# t = F_{6n}, x = F_{6n+1}, y = F_{6n+2}
def cnt6b t "?msd_fib $fib6(t) & Eu $adjfib(u,t) & 
   F[(u-1)/4+j+k]=@1 & j<t/2 & k<t/2":
# t = F_{6n}, u = F_{6n-1}
def numfib "?msd_fib Ey $adjfib(x,y) & i<=y & $isfib(i)":
def numfib6 t "?msd_fib $fib6(t) & $numfib(i,t)":
\end{verbatim}
Here
\begin{itemize}
\item {\tt fib6} accepts the Zeckendorf expansion of $F_{6n}$ for $n \geq 1$.
\item {\tt cnt6a} generates a linear representation of rank $123$ for the function $f(t)$.
\item {\tt cnt6b}
generates a linear representation of rank $121$ for the function $g(t)$.
\item {\tt numfib}$(i,x)$ accepts if and only
if $x = F_n$ for some $n \geq 2$ and $i=n$.
\item {\tt numfib6} creates a linear representation of rank $14$ for the function $h(t)$.
\end{itemize}
In {\tt Walnut} these linear representations are provided as {\tt Maple} files.

From these linear representations we can easily compute, through standard matrix methods, a linear representation of rank $258 = 123 + 121 + 14$ for
$3g(t) - 3f(t) - h(t)$; for the details, see \citet[\S 4.10]{Shallit:2023}. There is an algorithm to decide
if a given linear representation represents the $0$ function \citep{Berstel&Reutenauer:2011}.  We use {\tt Walnut} to create the representations and then a {\tt Maple} program\footnote{Available at \url{https://cs.uwaterloo.ca/~shallit/papers.html}.} to check if $3f(t)-3g(t)-h(t)$ is identically $0$. 
This proves Eq.~\eqref{t1}, and Eq.~\eqref{t2} follows in a similar manner.
\end{proof}

\begin{corollary}
The number of distinct values of $T(i, F_{3n}/2, F_{3n}/2)$ is
$\Theta(n)$.
\end{corollary}
\begin{proof}
From a result of
\cite{Berthe} we know that the number of distinct values of
$T(i,x,x)$ is $O(\log x)$.  From Theorem~\ref{bnd} we know the number of distinct values of $T(i, F_{3n}/2, F_{3n}/2)$ is $\Omega(n)$.  Combining these gives the result.
\end{proof}

\section{Balanced rectangles more generally}

Let us look at the notion of balanced rectangles somewhat more generally.   Given a word
$w$ and a symbol $c$, we let
$|w|_c$ denote the number of occurrences of $c$ in the word $w$.   We say that an infinite sequence $\bf a$ is $r$-balanced if, for all $c$ and all pairs of equal-length blocks $x$ and $y$ occurring in $\bf a$, we have
$\left| |x|_c - |y|_c \right| \leq r$.

We can generalize this to word rectangles in the obvious way:  if $A(i,m,n)$ is defined as in
Eq.~\eqref{aimn}, then 
$|A(i,m,n)|_c$ is the total number of occurrences of the letter $c$ in the $m \times n$ rectangle $A(i,m,n)$.  Then the natural question is,
for which pairs $(m,n)$ are all the rectangles $A(i,m,n)$ $r$-balanced with respect to each other? That is, for which $m,n$ do we have
$\left| |A(i,m,n)|_c - |A(j,m,n)|_c \right| \leq r$
for all $i, j$ and all $c$?

Consider the infinite Tribonacci word ${\bf TR} = 0102010\cdots = t_0 t_1 t_2 \cdots$, the fixed point of the morphism
$0 \rightarrow 01$, $1 \rightarrow 02$, and $2 \rightarrow 0$; see, for example, \cite{Berstel:2002}.  It seems natural to examine the analogous question of balance for the word rectangles formed from this word.
It is known \cite[see][]{Richomme&Saari&Zamboni:2010} that the Tribonacci word is $2$-balanced; that is, for all $c \in \{0,1,2\}$ and $n \geq 0$, the number of occurrences of the letter $c$ in two words of length $n$ differs by at most $2$.  
Here is the problem we solve in this section:  characterize  $m,n$ such that all the $(m,n)$-rectangles built from the Tribonacci word are $2$-balanced with respect to each other.  By symmetry, it suffices to answer this question for $m \leq n$.

Our proof technique is, once again, based on automata.  Here we need automata that take the Tribonacci representation \citep{Carlitz&Scoville&Hoggatt:1972} of an integer as input.  However, the technique is somewhat different than the one we used for Fibonacci word rectangles.
\begin{theorem}
Assume $m \leq n$. Then
\begin{itemize}
    \item[(a)] If $m = 1$, then all the $m \times n$ rectangles are $2$-balanced;
    \item[(b)] If $m = 2$, there is a $77$-state Tribonacci automaton that accepts exactly those $n$ for which all the $m \times n$ rectangles are $2$-balanced;
    \item[(c)] If $m \geq 3$, then there are no $n$ for which the all the $m \times n$ rectangles are $2$-balanced.
\end{itemize}
\end{theorem}

\begin{proof}
\leavevmode
\begin{itemize}
    \item[(a)] Trivially follows from the $2$-balance of the Tribonacci word.
    \item[(b)] We use {\tt Walnut} to create the automaton.  The code for {\tt triba}, {\tt tribb}, etc.\ is taken from \citet[\S 10.12]{Shallit:2023}.
\begin{verbatim}
reg shift {0,1} {0,1} "([0,0]|[0,1][1,1]*[1,0])*":

def triba "?msd_trib (s=0&n=0) | Ex $shift(n-1,x) & s=x+1":
# position of the n'th 0 in TR, indexing starting with 1
def tribb "?msd_trib (s=0&n=0) | Ex,y $shift(n-1,x) &
   $shift(x,y) & s=y+2":
# position of the n'th 1 in TR, indexing starting with 1
def tribc "?msd_trib (s=0&n=0) | Ex,y,z $shift(n-1,x) &
   $shift(x,y) & $shift(y,z) & s=z+4":
# position of the n'th 2 in TR, indexing starting with 1
def tribd "?msd_trib Et,u 
   $triba(s,t) & $triba(s+1,u) & t<=n & n<u":
# s = number of 0's in TR[0..n-1]
def tribe "?msd_trib Et,u 
   $tribb(s,t) & $tribb(s+1,u) & t<=n & n<u":
# s = number of 1's in TR[0..n-1]
def tribf "?msd_trib Et,u 
   $tribc(s,t) & $tribc(s+1,u) & t<=n & n<u":
# s = number of 2's in TR[0..n-1]

def num0 "?msd_trib Ex,y $tribd(i,x) & $tribd(i+n,y) & z+x=y":
# z is the number of 0's in TR[i..i+n-1]
def num1 "?msd_trib Ex,y $tribe(i,x) & $tribe(i+n,y) & z+x=y":
# z is the number of 1's in TR[i..i+n-1]
def num2 "?msd_trib Ex,y $tribf(i,x) & $tribf(i+n,y) & z+x=y":
# z is the number of 2's in TR[i..i+n-1]

# what follows is the code that generates the automaton 
# accepting the n such that the 2xn rectangles are 2-balanced

def row2_0 "?msd_trib Ex,y 
   $num0(i,n,x) & $num0(i+1,n,y) & z=x+y":
# z is number of 0's in 2xn rectangle cornered at i
def row2_1 "?msd_trib Ex,y 
   $num1(i,n,x) & $num1(i+1,n,y) & z=x+y":
# z is number of 1's in 2xn rectangle cornered at i
def row2_2 "?msd_trib Ex,y 
   $num2(i,n,x) & $num2(i+1,n,y) & z=x+y":
# z is number of 2's in 2xn rectangle cornered at i

def chk2_0 "?msd_trib Ai,j,y,z 
   ($row2_0(i,n,y) & $row2_0(j,n,z) & y<=z) => (z<=y+2)":
# check 2-balance of 0's in 2xn matrices
def chk2_1 "?msd_trib Ai,j,y,z 
   ($row2_1(i,n,y) & $row2_1(j,n,z) & y<=z) => (z<=y+2)":
# check 2-balance of 1's in 2xn matrices
def chk2_2 "?msd_trib Ai,j,y,z 
   ($row2_2(i,n,y) & $row2_2(j,n,z) & y<=z) => (z<=y+2)":
# check 2-balance of 2's in 2xn matrices

def bal2 "?msd_trib $chk2_0(n) & $chk2_1(n) & $chk2_2(n)":
# 77 states    
\end{verbatim}
The first few such $n$ are
$$ 1, 2, 3, 4, 7, 8, 9, 10, 11, 14, 15, 22, 23, 24, 27, 28, 33, 34, 35, 46, 47, 48, \ldots . $$
The fact that we need 77 states suggests there is unlikely to be a simple English description of these $n$.

\item[(c)]  Once again we use {\tt Walnut}, but with a different idea.  Define
the matrix $B(i,m,n)$ to be $A(i,m,n)$ with the letters
recoded as follows:  $0, 1 \rightarrow 0$ and $2 \rightarrow 2$.
The idea is to show that provided $m, n \geq 3$ there exist $i,j$
such that $B(i,m,n)$ and $B(j,m,n)$ agree, except on the $3 \times 3$ submatrix lying in the lower right corner:  $B(i,m,n)$ has
$$ \left[ \begin{matrix}
0 & 0 & 2 \\
0 & 2 & 0 \\
2 & 0 & 0
\end{matrix} \right]$$
in the lower right corner, while
$B(j,m,n)$ has
$$ \left[ \begin{matrix}
0 & 0 & 0 \\
0 & 0 & 0 \\
0 & 0 & 0
\end{matrix}\right] .$$
Thus $|A(i,m,n)|_2 \geq |A(j,m,n)|_2 + 3$, showing that
the $m \times n$ rectangles cannot be $2$-balanced with respect to each other.

We use the following {\tt Walnut} code:
\begin{verbatim}
morphism map2 "0->0 1->0 2->2":
image TR2 map2 TR:
def trib2eqfac "?msd_trib Au,v (u>=i & u<i+n & u+j=v+i) =>
   TR2[u]=TR2[v]":
eval cb2 "?msd_trib Ap Ei,j $trib2eqfac(i,j,p) & TR2[i+p]=@0 & 
   TR2[i+p+1]=@0 & TR2[i+p+2]=@0 & TR2[i+p+3]=@0 & 
   TR2[i+p+4]=@0 & TR2[j+p]=@0 & TR2[j+p+1]=@0 & 
   TR2[j+p+2]=@2 & TR2[j+p+3]=@0 & TR2[j+p+4]=@0":
\end{verbatim}
And {\tt Walnut} returns {\tt TRUE}.

Here {\tt TR2} is an automaton for the Tribonacci word ${\bf TR}_2$ that is 0 at positions where the Tribonacci word has values 0 or 1, and 2 at positions with value 2.
The automaton {\tt trib2eqfac} accepts $(i,j,n)$ if and only if ${\bf TR}_2[i..i+n-1]={\bf TR}_2[j..j+n-1]$.

Finally, {\tt cb2} says that for all $p$ there exist $i, j$
such that the block of length $p+5$ starting at $i$ looks like $w00000$, while
the block of length $p+5$ starting at $j$ looks like $w00200$.
Thus, if $m, n \geq 3$, we can set $p = m+n-6$ to find two $m \times n$ rectangles that are unbalanced with respect to each other for the symbol $2$. \vspace{-1.7\baselineskip}
\end{itemize}
\end{proof}

\section{The Thue--Morse word}

We can also form rectangles out of the
Thue--Morse word
${\bf t} = 01101001\cdots$, fixed point of the morphism $0 \rightarrow 01$, $1 \rightarrow 10$.
Here we can get a finite automaton for a complete description of the balance.  It is easier in this case because $t_{2n+1} = 1-t_{2n}$, so one can ``pair up'' entries except at the edges of the rectangle.

\begin{theorem}
\leavevmode
\begin{itemize}
\item[(a)] Every $m \times n$ word rectangle $A(i,m,n)$  built from $\bf t$ satisfies $|2 \cdot |A(i,m,n)|_1 -mn| \leq 4$.
\item[(b)] If there is an $m \times n$ word rectangle with
$2|A(i,m,n)|_1 - mn = c$, then there is one with
$2|A(j,m,n)|_1 - mn = -c$.
\item[(c)] The $m \times n$ word rectangles, for $m, n\geq 1$ have balance
either $1,2,3$ or $4$.
\end{itemize}
\end{theorem}

\begin{proof}
Let us write ${\bf t} = 01101001\cdots = t_0t_1t_2\cdots$.  (This is the same notation we used for Tribonacci but hopefully the reader will not be confused.) It follows from an easy induction that $t_{2k} + t_{2k+1} = 1$ for all $k\geq 0$. Similarly, $t_{4k} + t_{4k+2} = 1$ and $t_{4k+1} + t_{4k+3}=1$ for all $k\geq 0$. 
Recall the definition~\eqref{aimn} of $A(i,m,n)$. Since each row contains a block of ${\bf t}$, we can use the fact $t_{2k} + t_{2k+1} = 1$ to pair up adjacent bits with sum $1$ in each row of $A(i,m,n)$. Then exactly the following bits will be left unmatched: all entries in the leftmost column with odd indices and all entries in the rightmost column with even indices.
In the second step we can use $t_{4k+1} + t_{4k+3}=1$ to pair up bits in the leftmost column, and use $t_{4k} + t_{4k+2} = 1$ to pair up bits in the rightmost column.
At the end, at most four bits might be left unmatched, namely
\begin{itemize}
\item either $a_i$ or $a_{i+1}$, in case one of them is congruent to $3$ modulo $4$;
\item either $a_{i+m-2}$ or $a_{i+m-1}$, in case one of them is congruent to $1$ modulo $4$;
\item either $a_{i+n-1}$ or $a_{i+n}$, in case one of them is congruent to $2$ modulo $4$;
\item either $a_{i+n+m-3}$ or $a_{i+n+m-2}$, in case one of them is congruent to $0$ modulo $4$.
\end{itemize}
Thus we have matched $mn-\ell$ bits in total, with sum $(mn - \ell)/2$, and $\ell$ additional unmatched bits.
Since $0 \leq \ell \leq 4$,
this immediately implies $mn/2 - 2 \leq T(i,m,n) \leq mn/2 - 2 + 4 = mn/2 + 2$.
This gives the bound in (a).

(b) follows because if there is a word rectangle with $2|A(i,m,n)|_1 - mn = c$, then there is a word rectangle with every 0 flipped to 1 and every 1 flipped to 0,
giving a difference of $-c$.

(c) follows from (a) and the fact that the Thue--Morse sequence is not eventually periodic, so we cannot have balance 0.
\end{proof}

Now let us create a {\tt Walnut} automaton to compute the balance.  This introduces some complications.  Possibly  the conceptually simplest way involves computation with negative numbers, and the simplest way to do that in {\tt Walnut} is to represent integers in base-$(-2)$.  Then, at the end, when getting the automaton for balance, we
can convert it to a base-$2$
automaton using the techniques of \cite{Shallit&Shan&Yang:2023}.

\begin{theorem}
    There is an automaton of $92$ states that on input $m$ and $n$ in base $2$ determines the balance of the $m \times n$ rectangles
    built from $\bf t$.
\end{theorem}

\begin{proof}
We use the following {\tt Walnut} code.
The idea is illustrated for the case $i,m,n$ all even.  We consider pairing all the entries of the matrix with an even index with the odd index that follows it.
This leaves unpaired entries in the leftmost column and the
rightmost column.

In the leftmost column these entries are
$t_{i+1} + t_{i+3} + \cdots + t_{i+m-1}$.
Since all the indices are odd,
this is the same as
$(1-t_{i/2}) +  (1-t_{(i+2)/2}) + \cdots + (1-t_{(i+m-2)/2})
= m/2 - (t_{i/2} + t_{i/2+1} + \cdots + t_{(i+m)/2-1})$.
In the rightmost column these entries are
$t_{i+n} + t_{i+n+2} + \cdots + t_{i+m+n-2} =
t_{(i+n)/2} + t_{(i+n)/2 + 1} + \cdots +
t_{(i+m+n)/2 -1}$.
The case of $i$ odd is similar, and we reduce the other cases to this one by removing a row or column, or both.

Here the automaton {\tt TM2} computes the Thue-Morse sequence, but with input given in base-$(-2)$.  The automaton
{\tt twonegtwo} takes two inputs,
$w$ and $x$ in parallel, and accepts if $w$ considered as a base-$(-2)$ expansion equals $x$ considered as a base-$2$ expansion.  Both can be downloaded from the website of the first author.
\begin{verbatim}
def n2even "?msd_neg_2 Ek n=2*k":
def n2odd "?msd_neg_2 Ek n=2*k+1":
def prefsum "?msd_neg_2 (s=n/2 & $n2even(n)) | 
   (s=n/2+TM2[n-1] & $n2odd(n))":
# sum of t[0..n-1]
# 17 states
def factorsum "?msd_neg_2 Eu,v $prefsum(i,u) & $prefsum(i+n,v) & 
   s+u=v":
# sum of t[i..i+n-1]
# 317 states
def m_even_n_even "?msd_neg_2 Ea,b $factorsum(i/2,m/2,a) & 
   $factorsum((i+n)/2,m/2,b) & s=2*(b-a)":
# s = 2T(i,m,n) - mn  for m even, n even

def m_even_n_odd "?msd_neg_2 Ea,b $m_even_n_even(i,m,n-1,a) & 
   $factorsum(i+n-1,m,b) & s=a+2*b-m":
def m_odd_n_even "?msd_neg_2 Ea,b $m_even_n_even(i,m-1,n,a) & 
   $factorsum(i+m-1,n,b) & s=a+2*b-n":
def m_odd_n_odd "?msd_neg_2 Ea,b $m_odd_n_even(i,m,n-1,a) & 
   $factorsum(i+n-1,m,b) & s=a+2*b-m":

def tmt "?msd_neg_2  i>=0 & m>=0 & n>=0 & 
(($n2even(m) & $n2even(n) & $m_even_n_even(i,m,n,s)) |
($n2even(m) & $n2odd(n) & $m_even_n_odd(i,m,n,s)) |
($n2odd(m) & $n2even(n) & $m_odd_n_even(i,m,n,s)) |
($n2odd(m) & $n2odd(n) & $m_odd_n_odd(i,m,n,s)))":
# computes 2T(i,m,n) - mn for all i, m, n >= 0
# 3678 states

def maxval "?msd_neg_2 (Ei $tmt(i,m,n,x)) & 
   Ai,s $tmt(i,m,n,s) => s<=x":
# 81 states

def tmbal "Ex,y,b $twonegtwo((?msd_neg_2 x),m) & 
   $twonegtwo((?msd_neg_2 y),n) & $twonegtwo((?msd_neg_2 b),z) &
   $maxval(?msd_neg_2 x,y,b)":
# 94 states

def aut0 "$tmbal(m,n,0)":
def aut1 "$tmbal(m,n,1)":
def aut2 "$tmbal(m,n,2)": 
def aut3 "$tmbal(m,n,3)":
def aut4 "$tmbal(m,n,4)":

combine TMB aut0=0 aut1=1 aut2=2 aut3=3 aut4=4:
# result has 92 states 
\end{verbatim}\vspace{-\baselineskip}
\end{proof}

\begin{remark}
This was a fairly large computation in {\tt Walnut}.  The largest intermediate automaton in the computation of {\tt maxval} had 32,465,028 states.  The computation
took 12834 seconds of CPU
time and required 100 gigabytes of RAM.
\end{remark}

\begin{theorem}
Suppose $m, n\geq 3$. The set of $m \times n$ word rectangles for $\bf t$ has balance exactly $3$ if and only if $m$ and $n$ are both odd.  Here by ``balance exactly $3$'' we mean the balance is bounded above by $3$, and there is also a pair of rectangles that achieves this bound.
\end{theorem}

\begin{proof}
We use the following {\tt Walnut} code:
\begin{verbatim}
def even "Ek n=2*k":
def odd "Ek n=2*k+1":
eval bal3 "Am,n (m>=3 & n>=3) => 
   (($odd(m) & $odd(n)) <=> TMB[m][n]=@3)":
\end{verbatim}
And {\tt Walnut} returns {\tt TRUE}.
\end{proof}

The $m \times n$ rectangles having balance exactly $2$ or $4$ are more
difficult to understand.

\section{Final remarks}

Recently the second author has given criteria, in terms of the Ostrowski representation of any real number $\alpha$, for deciding whether $(m,n)$ is balanced for the matrix ${\cal A}_{\alpha}$; see \cite{Vukusic:2026}.

We acknowledge with thanks discussions with Pierre Popoli and Narad Rampersad.  We also thank the referees for their helpful remarks.

All the {\tt Walnut} code and {\tt Maple} files used can be downloaded from \\
\centerline{\url{https://cs.uwaterloo.ca/~shallit/papers.html} . } 


\nocite{*}
\bibliographystyle{abbrvnat}
\bibliography{abbrevs,refs}
\label{sec:biblio}

\appendix
\section{{\tt Walnut} commands}

The three principal {\tt Walnut} commands used in this paper are
\begin{itemize}
    \item {\tt reg}:  defines a regular expression;
    \item {\tt def}: defines an automaton for later use;
    \item {\tt eval}: evaluates a first-order logical statement with no free variables as {\tt TRUE} or {\tt FALSE}.
\end{itemize}

The jargon {\tt ?msd\_fib} at the beginning of a command tells {\tt Walnut} that numbers are represented in the Zeckendorf numeration system.   Similarly, {\tt ?msd\_trib} at the beginning of a command tells {\tt Walnut} that numbers are represented in the Tribonacci numeration system, and
{\tt ?msd\_neg\_2} tells {\tt Walnut} to use base-$(-2)$.

Logical AND $(\wedge)$ is represented by {\tt \&}; logical OR $(\vee)$ is represented by {\tt |}; logical NOT $(\neg)$ is represented by
{\tt \char'176}; implication is represented by {\tt =>}; and logical
IFF is represented by {\tt <=>}.

The capital letter {\tt A} is the universal quantifier $\forall$; the capital letter {\tt E} is the existential quantifier $\exists$.

Comments are denoted with the symbol {\tt \#}.

When an automaton is defined as a function of its unbound variables, it may be reused and the order of the arguments is the alphabetical order of the names of the free variables that were used when it was defined.

The capital letter {\tt F} refers to the Fibonacci word ${\bf f} = 01001\cdots$, the capital letters {\tt TR} 
refer to the Tribonacci word
${\bf TR} = 0102010\cdots$,
and {\tt T} refers to the Thue--Morse word ${\bf t} = 01101001\cdots$.
In each the word is indexed starting at index $0$.

The command {\tt morphism} constructs an automaton from a morphism, and {\tt image} applies a coding to the transitions.

\end{document}